\documentclass[a4paper,12pt]{amsart}
\usepackage{amsmath,amsfonts,amsthm,amssymb}
\usepackage[utf8]{inputenc}
\usepackage{hyperref}

\title{A structure theorem for sets of small popular doubling}
\author{Przemys\l aw Mazur}
\address{Mathematical Institute, Radcliffe Observatory Quarter, Woodstock Road, Oxford OX2 6GG, United Kingdom}
\email{przemyslaw.mazur@maths.ox.ac.uk}

\newtheorem{thm}{Theorem}[section]
\newtheorem{prp}[thm]{Proposition}
\newtheorem{lm}[thm]{Lemma}

\theoremstyle{definition}

\newtheorem{cor}[thm]{Corollary}

\renewcommand{\le}{\leqslant}
\renewcommand{\ge}{\geqslant}
\renewcommand{\epsilon}{\varepsilon}
\renewcommand{\phi}{\varphi}

\begin{document}
 
 \begin{abstract}
  In this paper we prove that every set $A\subset\mathbb{Z}$ satisfying the  inequality $\sum_{x}\min(1_A*1_A(x),t)\le(2+\delta)t|A|$ for $t$ and $\delta$ in suitable ranges, then $A$ must be very close to an arithmetic progression. We use this result to improve the estimates of Green and Morris for the probability that a random subset $A\subset\mathbb{N}$ satisfies $|\mathbb{N}\setminus(A+A)|\ge k$; specifically we show that $\mathbb{P}(|\mathbb{N}\setminus(A+A)|\ge k)=\Theta(2^{-k/2})$.
 \end{abstract}
 
 \subjclass[2010]{11P70}

 \keywords{small popular doubling, structure theorem, coset, progression, regularity lemma}
 
 \maketitle
 
 \section{Introduction}
 
 Let us start with recalling Freiman $(3k-3)$ Theorem. It states that every finite subset $A\subset\mathbb{Z}$ satisfying $|A+A|<3|A|-3$ is contained in an arithmetic progression of length $|A+A|-|A|+1$. Comparing this with a lower bound $|A+A|\ge 2|A|-1$ valid for all nonempty finite subsets of $\mathbb{Z}$, we can see that this result describes sets for quite large range of values of $|A+A|$. Our goal is to give a similar result for a set with a few \emph{popular} sums. Note that it cannot be done dirctly; the reason is that the set $S_k(A)=\{x\in\mathbb{Z}:|A\cap (x-A)|\ge k\}$ of $k$-popular sums is empty if $k\ge 3$ and $A$ is a highly independent set. Instead, we need to consider a different quantity, namely the \emph{average} size of $S_k$ for $1\le k\le t$, which also appeared quite natural to Pollard in his work \cite{pollard-74} back in 1974.
 
 At this point it is convenient to use the notation of convolution. From now on, we will consider any abelian group $G$ to be equipped with the couning measure, which leads to the definition
 \begin{equation*}
  f*g(x)=\sum_{y\in G} f(y)g(x-y)
 \end{equation*}
 for any functions $f,g:G\to\mathbb{C}$ for which the above expression makes sense (i.e. is absolutely convergent; we will use it mostly for $f,g$ being indicator functions of finite sets). Having this notation, we can restate Pollard's theorem as
 \begin{equation*}
  \sum_{x\in\mathbb{Z}/p\mathbb{Z}} \min(1_A*1_B(x),t)\ge \min(|A|\cdot |B|,\ t\cdot\min(p,|A|+|B|-t))
 \end{equation*}
 for any prime $p$ and sets $A,B\subset\mathbb{Z}/p\mathbb{Z}$. It is not hard to prove the corresponding statement for subsets of the integers (and even easier to deduce it from Pollard's theorem); in particular, for a single set $A\subset\mathbb{Z}$ and an integer $0\le t\le |A|$ we have $\sum_{x\in\mathbb{Z}}\min(1_A*1_A(x),t)\ge t(2|A|-t)$. One can also prove (or deduce from Vosper's theorem \cite{vosper-56}, a corresponding statement for $\mathbb{Z}/p\mathbb{Z}$) that the only sets for which we have the equality in the above inequality are arithmetic progressions.
 
 Our goal is to extend this structure theorem to be able to recognize sets $A\subset\mathbb{Z}$ satisfying $\sum_{x\in\mathbb{Z}}\min(1_A*1_A(x),t)\le (2+\delta)t|A|$ for suitable ranges of parameters $t$ and $\delta$ as those that can be almost entirely covered by an arithmetic progression. Specifically, for $\frac{t}{|A|}\searrow 0$ we can pick $\delta$ as big as $\frac{1}{4}-O(\frac{t}{|A|})$. Note that we cannot expect all of $A$ to be covered by an arithmetic progression: a simple counterexample is an arithmetic progression with one extra point as far away as we like. The reader can check that for that set and $t$, $\frac{|A|}{t}$ sufficiently large the parameter $\delta$ can be as close to $0$ as we like, yet our set cannot be covered with an arithmetic progression of bounded length.
 
 In the next sections we use this result to slightly modify the regularity lemma proven by Green and Morris in \cite{green-15}, which allows us to improve the estimates on the probability that a sumset of a random subset $A$ of natural numbers misses at least $k$ elements. More precisely, we will show not only that the sequence $p_k=2^{k/2}\cdot\mathbb{P}(|\mathbb{N}\setminus (A+A)|\ge k)$ is bounded, but also that it is increasing (and therefore convergent) along indices of the same parity (i.e. odd or even).
 
 Before proceeding, let us state precisely the results to be proven.
 
 \begin{thm}\label{pop}
  Let $S\subset\mathbb{Z}$ be a set of size $N>0$ and let $t$ be a positive integer. Suppose that
  \begin{equation*}
   \sum_{x\in\mathbb{Z}}\min(1_S*1_S(x),t)\le (2+\delta)Nt,
  \end{equation*}
  for some $\delta>0$. Then there is an arithmetic progression $P$ with of length at most $(1+2\delta)N+6t$ containing all but at most $\frac{5t}{2}$ points of $S$, provided that $\delta+\frac{5t}{N}\le\frac{1}{4}.$
 \end{thm}
 
 \begin{thm}\label{prob}
  Let $A\subset\mathbb{N}$ be a set chosen randomly by picking each element of $\mathbb{N}$ independently with probability $\frac{1}{2}$. Define a sequence $\{p_k\}$ via
  \begin{equation*}
   p_k=2^{k/2}\cdot\mathbb{P}(|\mathbb{N}\setminus (A+A)|\ge k).
  \end{equation*}
  Then the subsequences $\{p_{2k}\}$ and $\{p_{2k+1}\}$ are both increasing and bounded and therefore convergent. In particular $\mathbb{P}(|\mathbb{N}\setminus (A+A)|\ge k)=\Theta(2^{-k/2})$ and the implied constants can only oscillate between $c$ and $c\sqrt{2}$ for some $c>0$ as $k\to\infty$.
 \end{thm}

 \section{Wrapping argument}

 In this section we start proving Theorem \ref{pop} with similar methods that Lev and Smeliansky used in \cite{lev-95} to prove (a generalisation of) Freiman's $3k-3$ Theorem. More precisely, their first step was to wrap the set $S$ modulo $q:=(\max S-\min S)$ and consider a subset of a finite group instead. Note that since the sets $(S+\min S)$ and $(S+\max S)$ share only one element, this wrapping procedure results in a huge decrement in the doubling constant. In our situation taking just two endpoints would be too careless to achieve good results; luckily we can still find two points near the ends such that wrapping modulo their difference gives us what we need.
 
  \begin{prp}\label{wrap}
  Let $S\subset\mathbb{Z}$ be a finite set and suppose that $N:=|S|>0$. Let $t$ be a positive integer satisfying $2t<N$ and suppose that 
  \begin{equation*}
   \sum_{x\in\mathbb{Z}}\min(1_S*1_S(x),t)\le (2+\delta)Nt
  \end{equation*}
  for some $\delta\ge 0$. Then there exist a positive integer $n$ and an integer $x$ and such that the set $S'=\textup{(}S\cap[x,x+n)\textup{)}\pmod{n}$ (the image of the set $S\cap[x+n)$ under the projection $\pmod{n}$) satisfies the following conditions:
  \begin{itemize}
   \item $|S'|\ge N-2t$,
   \item $\sum_{x\in\mathbb{Z}/n\mathbb{Z}}\min(1_S'*1_S'(x),t)\le \left(1+2\delta+\frac{6t}{N}\right)Nt$.
  \end{itemize}

 \end{prp}
 
 Let us make the remark that in the final statement the parameter $\delta$ comes with coefficient $2$, which leads to some limitations in the statement of Theorem \ref{pop}, such as $\delta<\frac{1}{4}$. We believe that this argument can be performed in the way that would give the coefficient $1$, which would extend the range of $\delta$ up to $\frac{1}{2}$ and consequently allow us to prove the corresponding statement in a finite group of prime order using similar methods (for that we need to let $\delta>6-4\sqrt{2}>\frac{1}{3}$).
 
 \begin{proof}
  Divide $S$ into three subsets $A,B,C$ with $|A|=|C|=t$, $|B|=N-2t$, $\max A<\min B$, $\max B<\min C$ (intuitively they are the left, middle and right part respectively). Let $f,g,h$ be the corresponding indicator functions (i.e. $f=1_A$, $g=1_B$, $h=1_C$). Substiuting it to the convolution we get
  \begin{equation*}
   1_S*1_S=(f+g+h)*(f+g+h)\ge 2(f*g+g*h),
  \end{equation*}
  where the inequality comes from discarding some of the positive summands. Note that since $\max(A+B)<\max(B+C)$, the functions $f*g$ and $g*h$ are supported on disjoint sets and therefore the above implies the following:
  \begin{multline*}
   \sum_{x\in\mathbb{Z}}\min(2f*g(x),t)+\sum_{x\in\mathbb{Z}}\min(2g*h(x),t)\le\\
   \;e\sum_{x\in\mathbb{Z}}\min(1_S*1_S(x),t)\le(2+\delta)Nt.
  \end{multline*}
  Now we use an easy to check inequality $s^2\ge t(2s-\min(2s,t))$, valid for all real numbers $s$. Speciffically, we substitute $s=f*g(x)$ and $s=g*h(x)$ for all $x\in\mathbb{Z}$ and add them together to get
  \begin{equation*}
   \sum_{x\in\mathbb{Z}}(f*g)^2(x)+\sum_{x\in\mathbb{Z}}(g*h)^2(x)\ge t(4t(N-2t)-(2+\delta)Nt)=(2-\delta)Nt^2-8t^3.
  \end{equation*}
  Here we also used the previous estimate and the formula for the sum $\sum_{x\in\mathbb{Z}}f*g(x)=\sum_{x\in\mathbb{Z}}g*h(x)=t(N-2t)$. Note that since each $x\in\mathbb{Z}$ can be written in exactly $f*g(x)$ ways as a sum $x=a+b$ for $a\in A$, $b\in B$ (and similarly for $g*h$), the above expression is in fact equal to
  \begin{equation*}
   \sum_{x\in\mathbb{Z}}(f*g)^2(x)+\sum_{x\in\mathbb{Z}}(g*h)^2(x)=\sum_{a\in A}\sum_{b\in B} f*g(a+b)+\sum_{c\in C}\sum_{b\in B}g*h(b+c).
  \end{equation*}
  By choosing elements $a\in A$ and $c\in C$ for which the inner sums are above average, we can see that
  \begin{equation*}
   \sum_{b\in B} f*g(a+b)+\sum_{b\in B}g*h(b+c)\ge(2-\delta)Nt-8t^2
  \end{equation*}
  for some $a$ and $c$ (as $|A|=|C|=t$). SInce $f*g$ and $g*h$ are bounded by both $1_S*1_S$ and $t$, we actually proved that
  \begin{equation*}
   \sum_{x\in(a+B)\cup (c+B)}\min(1_S*1_S(x),t)\ge(2-\delta)Nt-8t^2.
  \end{equation*}
  This is just enough for us do define the wrapping procedure. Indeed, projection modulo $(c-a)$ merges the sets $a+B$ and $c+B$ into a single copy of $B$, on which sum of the values of the above function cannot exceed $t|B|=t(N-2t)$. After a short calculation $(2+\delta)Nt-((2-\delta)Nt-8t^2)+t(N-2t)=(1+2\delta)Nt+6t^2$ we see that the set $S'$ for $x=a$ and $n=c-a$ satisfies the inequality
  \begin{equation*}
   \sum_{x\in\mathbb{Z}}\min(1_{S'}*1_{S'}(x),t)\le(1+2\delta)Nt+6t^2.
  \end{equation*}
 \end{proof}

 After proving this our goal is to show that the set $S'$ is close to a coset of a subgroup of $\mathbb{Z}/n\mathbb{Z}$ which would correspond to a progression in $\mathbb{Z}$. We will deal with that problem in the next section.
 
 \section{Popular doubling less than $\frac{3}{2}$}

 The next step of the proof by Lev and Smeliansky was to use Kneser's Theorem stating that for any finite subsets $A,B$ of an abelian group $G$ the subgroup of all elements $h$ satisfying $A+B+h=A+B$ has cardinality at least $|A|+|B|-|A+B|$. The proof of that theorem requires checking a lot of scenarios and it is not clear how one could modify it to work for popular sums. On the contrary, the proof of a weaker statement that if $|A|=|B|=N$ and $|A+B|<\frac{3}{2}N$ then $A+B$ is a coset of a subgroup is much easier and, as it turns out, generalisable to our setting. Specifically, in this section we will proceed towards the following statement.
 
 \begin{prp}\label{doubling_3/2}
  Let $G$ be an abelian group and let $A,B\subset G$ be sets of size $N>0$. Let $t,\eta>0$ be two real numbers satisfying the inequality $\eta+\frac{t}{N}\le\frac{1}{2}$. Moreover suppose that the following inequality holds:
  \begin{equation*}
   \sum_{x\in G} \min(1_A*1_B(x),t)\le (1+\eta)Nt.
  \end{equation*}
  Then there exists a subgroup $H\le G$ and cosets $C_A$ and $C_B$ of $H$ satisfying the following conditions:
  \begin{itemize}
   \item $|H|\le (1+\eta)N$,
   \item $|A\setminus C_A|+|B\setminus C_B|\le t$.
  \end{itemize}
 \end{prp}
 
 Note that in this section we do not require $t$ to be an integer anymore. Let us also remark that we only need the above statement for the case $A=B$, but the proof of the more general case is not much harder so we decided to include it here. For convenience of the reader we split it into several lemmas. First of all we want to find the subgroup $H$. Although the statement of the following lemma appears to be new in the literature, the methods going into the proof were used in \cite{fournier-77}.
 
 \begin{lm}\label{subgroup}
  Let $G$ be an abelian group and let $A,B\subset G$ be sets of size $N>0$. Let $t,\eta>0$ be two real numbers satisfying the inequality $\eta+\frac{t}{N}\le\frac{1}{2}$. Moreover suppose that the following inequality holds:
  \begin{equation*}
   \sum_{x\in G} \min(1_A*1_B(x),t)\le (1+\eta)Nt.
  \end{equation*}
  Then there exists a subgroup $H\le G$ satisfying the following conditions:
  \begin{itemize}
   \item $1_A*1_{-A}(x)\ge(1-\eta)N$ and $1_B*1_{-B}(x)\ge(1-\eta)N$ for all $x\in H$,
   \item $1_A*1_{-A}(x)<2t$ and $1_B*1_{-B}(x)<2t$ for all $x\in G\setminus H$.
  \end{itemize}
 \end{lm}
 
 Before we proceed, let us observe that the triangle inequality for symmetric difference of sets $|V\triangle W|\le |U\triangle V|+|U\triangle W|$ can be rearranged as $|U\cap V|+|U\cap W|\le |U|+|V\cap W|$. We will frequently use a variant of this inequality, namely the assertion that $|U\cap (V-v)|+|U\cap (W-w)|\le |U|+|V\cap (W-w+v)|$ for various choices of $v,w$. We will refer to all kinds of this statement as triangle inequality.
 
 \begin{proof}
  Let us start with the set $D=\{x\in G:1_A*1_B(x)\ge t\}$. Note that $D$ contains most of the sums $a+b$; more precisely we have
  \begin{multline*}
   \#\{(a,b)\in A\times B:a+b\in D\}=\sum_{x\in D}1_A*1_B(x)=\\
   =\sum_{x\in D}t+\sum_{x\in G}\max(1_A*1_B(x)-t,0)=\\
   =t|D|+\sum_{x\in G}1_A*1_B(x)-\sum_{x\in G}\min(1_A*1_B(x),t)\ge t|D|+N^2-(1+\eta)Nt.
  \end{multline*}
  Now let $H=\{x\in G:1_A*1_{-A}(x)\ge 2t\}$. We would like to show that for any $h\in H$ we have $1_B*1_{-B}(x)\ge(1-\eta)N$. Let us start by noticing that for any $a\in |A\cap (A+h)|$ we have
  \begin{equation*}
   |B\cap(B+h)|\ge |(B+a)\cap D|+|(B+a+h)\cap D|-|D|
  \end{equation*}
  by triangle inequality. Now let $f:G\to [0,1]$ be an auxiliary function supported on $A\cap(A+h)$ and satisfying the condition $\sum_{x\in G}f(x)=2t$ (there exists one by choice of $h$). Multiplying the above inequality by $f(a)$ and adding them together we get
  \begin{equation*}
   2t|B\cap(B+h)|\ge\sum_{a\in A} (f(a)+f(a-h))|(B+a)\cap D|-2t|D|. 
  \end{equation*}
  Now we combine the inequalities
  \begin{align*}
   \sum_{a\in A} 2|(B+a)\cap D|&\ge2t|D|+2N^2-2(1+\eta)Nt, \\
   \sum_{a\in A} (2-f(a)-f(a-h))|(B+a)\cap D|&\le|B|\sum_{a\in A}(2-f(a)-f(a-h))=\\=N(2N-4t)
  \end{align*}
  to get
  \begin{multline*}
   |B\cap(B+h)|\ge\frac{2t|D|+2N^2-2(1+\eta)Nt-(2N^2-4Nt)-2t|D|}{2t}=\\
   =(1-\eta)N.
  \end{multline*}

  In similar way we can prove that the inequality $1_B*1_{-B}(x)\ge 2t$ implies $1_A*1_{-A}\ge(1-\eta)N$. Since $(1-\eta)N\ge 2t$, we have just constructed the set $H$ satisfying all the postulated inequalities. The only thing remaining is to show that $H$ is a subgroup. This follows from triangle inequality: if $h_1,h_2\in H$, then
  \begin{multline*}
   |A\cap(A+h_1-h_2)|\ge |A\cap(A+h_1)|+|A\cap(A+h_2)|-|A|\ge\\
   \ge 2(1-\eta)N-N=(1-2\eta)N\ge 2t
  \end{multline*}
  and consequently $h_1-h_2\in H$.
 \end{proof}

 Let us now turn for the moment to some estimates of the expressions of the form $\sum_{x\in G}\min(F(x),t)$.
 
 \begin{lm}
  Let $G$ be an abelian group and let $F:G\to [0,M]$ be a function satisfying $\sum_{x\in G}F(x)<\infty$. Then for any $t\in [0,M]$ we have
  \begin{equation*}
   \sum_{x\in G}\min(F(x),t)\ge\frac{t}{M}\sum_{x\in G}F(x).
  \end{equation*}
 \end{lm}

 \begin{proof}
  Notice that $\frac{t}{M}F(x)\le\min(t,F(x))$ for each individual $x$.
 \end{proof}

 \begin{cor}
   Let $G$ be an abelian group and let $F:G\to [0,+\infty)$ be a function satisfying $\sum_{x\in G}F(x)<\infty$. Then for any $t'>t>0$ we have
  \begin{equation*}
   \sum_{x\in G}\min(F(x),t)\ge\frac{t}{t'}\sum_{x\in G}\min(F(x),t').
  \end{equation*}
 \end{cor}

 \begin{proof}
  Just use the above lemma for $\min(F(x),t')$.
 \end{proof}

 \begin{cor}
  Let $G$ be an abelian group and let $A,B\subset G$ be sets of size $N>0$. Let $t,\eta>0$ be two real numbers satisfying the inequality $\eta+\frac{t}{N}\le\frac{1}{2}$. Moreover suppose that the following inequality holds:
  \begin{equation*}
   \sum_{x\in G} \min(1_A*1_B(x),t)\le (1+\eta)Nt.
  \end{equation*}
  Then for any $t'>t$ we have
  \begin{equation*}
   \sum_{x\in G} \min(1_A*1_B(x),t')\le (1+\eta)Nt'.
  \end{equation*}
 \end{cor}

 Let us go back to our considerations. We have already constructed a set $D$ containing most of the sums $a+b$ and the subgroup $H$ satisfying certain inequalities. Now it is time to construct a \emph{coset} $C$ of $H$ containing most of the sums $a+b$. We will do it in two steps: first we show that $C$ contains just enough sums to perform calculations quite accurately, which in turn will give us the cosets $C_A$ and $C_B$ with desired properties.
 
 \begin{lm}
  Let $G$ be an abelian group and let $A,B\subset G$ be sets of size $N>0$. Let $t,\eta>0$ be two real numbers satisfying the inequality $\eta+\frac{t}{N}\le\frac{1}{2}$. Moreover suppose that the following inequality holds:
  \begin{equation*}
   \sum_{x\in G} \min(1_A*1_B(x),t)\le (1+\eta)Nt.
  \end{equation*}
  Then there exists a coset $C$ of a subgroup $H$ satisfying $|C|\le (1+\eta)N$ and 
  \begin{equation*}
   \#\{(a,b)\in A\times B:a+b\in C\}>\frac{(1+\eta)N^2}{2}.
  \end{equation*}
 \end{lm}

 \begin{proof}
  Let $H$ be the subgroup constructed in Lemma \ref{subgroup}. We want to translate $H$ to make it contain most of the sums $a+b$, so a reasonable choice is to take $x_0\in G$ for which $1_A*1_B(x_0)$ is maximal and set $C=H+x_0$. Let $k=N-1_A*1_B(x_0)$. Note that by previous considerations for any $t'<\frac{N}{1+\eta}$ we have $\sum_{x\in G} \min(1_A*1_B(x),t')<N^2=\sum_{x\in G}1_A*1_B(x)$, which implies $1_A*1_B(x_0)\ge\frac{N}{1+\eta}$ or in other words $k\le\frac{\eta N}{1+\eta}$. Moreover, by triangle inequality we have $|1_A*1_B(x+x_0)-1_A*1_{-A}(x)|\le k$; in particular $C$ contains the set $C'=\{x\in G\ : 1_A*1_B(x)\ge 2t+k\}$. Therefore we are interested in the size of the set $\#\{(a,b)\in A\times B:a+b\in C'\}$. By the same calculations as for the set $D$ in the proof of the Lemma \ref{subgroup} we know that
  \begin{equation*}
   \#\{(a,b)\in A\times B:a+b\in C'\}\ge N^2-(1+\eta)N(2t+k)+(2t+k)|C'|.
  \end{equation*}
  Here we have also used here the previous corollary with $t'=2t+k$. Let us bound the size of $C'$ from below. We know that
  \begin{multline*}
   N^2=\sum_{x\in G}1_A*1_B(x)=\\
   =\sum_{x\in G}\min(1_A*1_B(x),2t+k)+\sum_{x\in C'}(1_A*1_B(x)-(2t+k))\le\\
   \le (1+\eta)N(2t+k)+|C'|(N-2t-2k),
  \end{multline*}
  which rearranges to $|C'|\ge\frac{N(N-(1+\eta)(2t+k))}{N-2t-2k}.$ Substituting this into the previous bound we get
  \begin{equation*}
   \#\{(a,b)\in A\times B:a+b\in C'\}\ge N^2-\frac{N(2t+k)(\eta N-(1+\eta)k)}{N-2t-2k}.
  \end{equation*}
  It is easy to check that the expression on the right hand side is a decreasing function in $t$, so we can substitute $t=(\frac{1}{2}-\eta)N$ to get
  \begin{equation*}
   \#\{(a,b)\in A\times B:a+b\in C'\}\ge N^2-\frac{N((1-2\eta)N+k)(\eta N-(1+\eta)k)}{2(\eta N-k)}.
  \end{equation*}
  Now it is also easy to check that this being greater than $\frac{1+\eta}{2}N^2$ is equivalent to the inequality $(\eta N-k)^2+(1-2\eta)\eta Nk+\eta k^2>0$, which is true because we assumed $\eta>0$.
  
  Now we only need to bound the size of $C$; because of triangle inequality it is contained in the set $C''=\{x\in G:1_A*1_B(x)\ge (1-\eta)N-k\}$, so using the corollary with $t'=(1-\eta)N-k>t$ we get
  \begin{multline*}
   t'|C|\le t'|C''|=\sum_{x\in C''}\min(1_A*1_B(x),t')\le\\
   \le\sum_{x\in G}\min(1_A*1_B(x),t')\le (1+\eta)Nt'.
  \end{multline*}
  Thus we have proved both desired inequalities.
 \end{proof}

 Now we are ready to prove Proposition \ref{doubling_3/2}.
 
 \begin{proof}[Proof of Proposition \ref{doubling_3/2}]
  Let $H$ and $C$ be as in previous considerations. By averaging argument, there exist $a\in A$, $b\in B$ with $|(A+b)\cap C|>\frac{1+\eta}{2}N$ and $|(B+a)\cap C|>\frac{1+\eta}{2}$. Define $C_A=C-b$ and $C_B=C-b$. Now let us estimate the sum of the expressions $\min(1_A*1_B(x),t)$ separately on and outside $C_A+C_B$. To do this, let $f,g,h:G\to[0,1]$ be auxiliary functions supported on $C_B$, $G\setminus C_B$ and $G\setminus C_A$ respectively, satisfying $f,g\le 1_B$, $h\le 1_A$ and the conditions
  \begin{align*}
   \sum_{x\in G}f(x)&=t,\\
   \sum_{x\in G}g(x)&=\min(|B\setminus C_B|,t),\\
   \sum_{x\in G}h(x)&=\min\left(|A\setminus C_A|,t-\sum_{x\in G}g(x)\right).
  \end{align*}
  Now we have the following estimates:
  \begin{align*}
   \sum_{x\in C_A+C_B}\min(1_A*1_B(x),t)&\ge\sum_{x\in C_A+C_B}1_{A\cap C_A}*f(x)>\frac{(1+\eta)Nt}2,\\
   \sum_{x\not\in C_A+C_B}\min(1_A*1_B(x),t)&\ge\sum_{x\not\in C_A+C_B}(1_{A\cap C_A}*g(x)+1_{B\cap C_B}*h(x))>\\
   &>\frac{(1+\eta)N}{2}\sum_{x\in G}(g(x)+h(x)).
  \end{align*}
 Comparing that to the initial estimate $\sum_{x\in G}\min(1_A*1_B(x),t)\le(1+\eta)Nt$ we see that $\sum_{x\in G}(g(x)+h(x))<t$, which is only possible if $\sum_{x\in G}g(x)=|B\setminus C_B|$ and $\sum_{x\in G}h(x)=|A\setminus C_A|$.
 Therefore $|A\setminus C_A|+|B\setminus C_B|<t$.
 
 To finish the proof, notice that $|H|=|C|\le(1+\eta)N$.
 \end{proof}

 Before proceeding, let us make the remark that in fact the larger the subgroup $H$ is, the less points of $A$ and $B$ are allowed to lie outside $C_A$ and $C_B$ respectively. One can try to perform even more precise calculations, using the fact that now we know that actually for all $x\in C=C_A+C_B$ we have $1_A*1_B(x)\ge t$. However we do not need it that much so we leave the result as it is.
 
 \section{Completing the proof}
 
 Having proved the results in the previous two sections, we are ready to prove Theorem \ref{pop}.
 
 \begin{proof}[Proof of Theorem \ref{pop}]
  Suppose that we have the set $S$ that for which the inequality $\sum_{x\in \mathbb{Z}}\min(1_S*1_S(x),t)\le (2+\delta)Nt$ holds. We use Lemma \ref{wrap} to obtain a set $S'\subset\mathbb{Z}/n\mathbb{Z}$ of size at least $N-2t$ satisfying the inequality $\sum_{x\in \mathbb{Z}}\min(1_{S'}*1_{S'}(x),t)\le (1+2\delta)Nt+6t^2$. We see that the assumptions of Proposition \ref{doubling_3/2} are satisfied with $A=B=S'$ as long as
  \begin{equation*}
   \frac{1}{2}\ge\frac{(1+2\delta)N+6t-|S'|}{|S'|}+\frac{t}{|S'|}=\frac{(1+2\delta)N+7t-|S'|}{|S'|},
  \end{equation*}
  in other words $3|S'|\ge (2+4\delta)N+14t$, which is certainly true if $\delta+\frac{5t}{N}\le\frac{1}{4}$.
  Propsition \ref{doubling_3/2} then tells us that the set $S'$ is essentially contained in a coset of a subgroup of size at most $(1+2\delta)N+6t$, with the exception of at most $\frac{t}{2}$ points. Unwrapping the situation back again, we see that the set $S$ has all but at most $\frac{5t}{2}$ elements contained in an arithmetic progression of length at most $(1+2\delta)N+6t$.
 \end{proof}

 \section{Regularity and counting sets with small sumset}
 
 This section is devoted to a lemma of Green and Morris on counting subsets of a cyclic group of prime order satisfying certain bounds on the size of the subset. Unfortunately we cannon just quote their result, as we need a slight modification of it. Therefore we need to move back to the statement of the regularity lemma, or more precisely, to \cite[Theorem 2.1]{green-15}, stated below.
 
 \begin{lm}[Green-Morris, regularity lemma]
  For every $\epsilon>0$, there exists $\delta=\delta(\epsilon)>0$ such that the following is true. Let $p>p_0(\epsilon)$ be a sufficiently large prime and let $A\subset\mathbb{Z}/p\mathbb{Z}$ be a set. There is a dilate $A^*=\lambda A$ and a prime $q$, $\frac{1}{\epsilon^{10}}\le q\le p^{1-\delta}$, such that the following holds. If
  $A^*_i=A^*\cap I_i(q)$ for each $i\in\mathbb{Z}/q\mathbb{Z}$ then, for at least $(1-\epsilon)q^2$ pairs $(i,j)\in(\mathbb{Z}/q\mathbb{Z})^2$
  \begin{equation*}
   \min(|A_i^*|,|A_j^*|)\le\epsilon p/q \qquad\text{or}\qquad|A_i^*+A_j^*|\ge(2-\epsilon)p/q.
  \end{equation*}
 \end{lm}
 Here we adopted the notation $I_i(q)=\{x\in\mathbb{Z}/p\mathbb{Z}:x/p\in[i/q,(i+1)/q)\}$. Note that $[i/q,(i+1)/q)+[j/q,(j+1)/q)\subset [(i+j)/q,(i+j+2)/q)$ as subsets of $\mathbb{R}/\mathbb{Z}$; intersecting those sets with $\mathbb{Z}/q\mathbb{Z}$ embedded in $\mathbb{R}/\mathbb{Z}$ in a natural way, we get the inclusion $I_i+I_j\subset I_{i+j}\cup I_{i+j+1}$. 
 
 It is now time to prove some bounds on the number of sets having fixed size and whose sumset has also fixed size. We cannot improve the bound given by Green and Morris; instead we will introduce a better bound for the number of \emph{exceptional} sets and at the same time use our result to prove that every non-exceptional set has certain structure. Specifically, we will proceed towards the proof of the following statement.
 
 \begin{prp}\label{counting}
  Let $\delta>0$ and $N>N_0(\delta)$ be a large natural number. For every $k,m\in\mathbb{N}$ satisfying $\delta N\le k\le N$, $2k-1\le m\le 2N-1$ the following statement is true. The family of all subsets $X\subset\{1,\ldots, N\}$ with $|X|=k$ and $|X+X|=m$ can be divided into two classes; one of them, the class of exceptional subsets, has cardinality at most $2^{\frac{m}{2}H(\frac{2k-1}{m})-\delta N}$, and each non-exceptional set (member of the other class) is almost contained in an arithmetic progression $P$ of length at most $(1+1200\delta)\frac{m}{2}$, so that we have $|(X+X)\setminus(P+P)|\le 24 \delta N$.
 \end{prp}
 
 For the definition of the function $H$, see the appendix.

 Before we proceed, let us make a few remarks. Firstly, our result is only valid in subsets of integers and not in the cyclic groups of prime order; the reason for that is Theorem \ref{pop} is of the same kind. Secondly, even in that case it does not improve the estimate of Green and Morris on the number of \emph{all} subsets with $|X|=k$ and $|X+X|=m$; however the additional structure allows us to prove Theorem \ref{prob}.
 
 Since we are following Green and Morris argument, we will also need Pollard's Theorem. It has already appeared in the introduction, but let us state once again in a somewhat more precise form.
 
 \begin{thm}[Pollard]
  Let $p$ be a prime number and let $A,B\subset\mathbb{Z}/p\mathbb{Z}$ be two sets. Let $t$ be an integer satisfying 
  \begin{equation*}
   \max(0, |A|+|B|-p)\le t\le\min(|A|,|B|).
  \end{equation*}
  Then the following inequality holds:
  \begin{equation*}
   \sum_{x\in\mathbb{Z}/p\mathbb{Z}}\min(1_A*1_B(x),t)\ge t(|A|+|B|-t).
  \end{equation*}
 \end{thm}
 
 Also, since Theorem \ref{pop} refers to the subsets of integers, we have to make it more compatible with the regularity lemma, which refers to the subsets of a finite group. To link those two statements, let us prove the following lemma.
 
 \begin{lm}\label{int-prog}
  Let $p$ be a prime and let $P,Q\subset\mathbb{Z}/p\mathbb{Z}$ be arithmetic progresssions satisfying $|P|\le\frac{p}{4}$ and $|P\cap Q|\ge\frac{|Q|}{2}+1$. Then the set $P\cap Q$ is an arithmetic progression with the same common difference as $Q$.
 \end{lm}

 \begin{proof}
  By dilating if necessary, we can assume that $P$ is an interval (i.e. the common difference of $P$ is $1$). Since $|P\cap Q|\ge\frac{|Q|}{2}+1$, we know that there are two consecutive elements of $Q$ that belong to $P$. That means that the common difference of $Q$ is less than the size of $P$; let us denote it by $d$. Suppose for the sake of contradiction, that the intersection $P\cap Q$ is not a progression of common difference $d$. In other words, if we look at the elements of $Q$ in order, we see at least two separate groups of elements of $P\cap Q$ with at least one element of $Q\setminus P$ in between. Since the common difference of $Q$ is $d$, each group of elements of $P\cap Q$ has cardinality at most $\left\lceil\frac{|P|}{d}\right\rceil$ and each group of elements of $Q\setminus P$ (maybe except those containing the endpoints) has cardinality at least $\left\lfloor\frac{p-|P|}{d}\right\rfloor$. Also, if we dentote $l=\left\lceil\frac{|P|}{d}\right\rceil$, then 
  \begin{equation*}
   \left\lfloor\frac{p-|P|}{d}\right\rfloor>\frac{p-|P|}{d}-1\ge\frac{3|P|}{d}-1>3\left\lceil\frac{|P|}{d}\right\rceil-4=3l-4,
  \end{equation*}
  so in fact $\left\lfloor\frac{p-|P|}{d}\right\rfloor\ge 3l-3$. By $d<|P|$ we know that $l\ge 2$. Now denote by $k\ge 2$ the number of groups of elements of $P\cap Q$, we see that the number of groups of $Q\setminus P$ not containing the endpoins is $k-1$. By assumption $Q\setminus P$ has at least $2$ elements less than $P\cap Q$, which leads to the inequality
  \begin{equation*}
   kl\ge |P\cap Q|\ge|Q\setminus P|+2\ge(k-1)(3l-3)+2.
  \end{equation*}
  Rearranging gives $(2k-3)(2l-3)\le -1$, which is impossible since both of the factors are positive.
 \end{proof}

 Now we are ready to prove Proposition \ref{counting}.
 
 \begin{proof}[Proof of Proposition \ref{counting}]
  Suppose that $\delta>0$ is a sufficiently small constant. Let $N>N_0(\delta)$ be a large natural number and let $p\in [8N,16N]$ be a prime. We consider each subset of $[N]$ to be a subset of $\mathbb{Z}/p\mathbb{Z}$ via the natural embedding $[N]\hookrightarrow\mathbb{Z}/p\mathbb{Z}$. Then for a subset $A\subset [N]$ with $|A|\ge\delta N$ we can use the regularity lemma with $\epsilon=2^{-7}\delta^4$ to obtain a dilate $A^*=\lambda A$, a prime number $q$ and a corresponding partition $A_i^*=A^*\cap I_i$. We can assume that $N_0$ is so large that it forces $q\le\delta^2 p$. We know that for at least $(1-\epsilon)q^2$ pairs $(i,j)\in(\mathbb{Z}/q\mathbb{Z})^2$, either
  \begin{equation*}
   \min(|A_i^*|,|A_j^*|)\le\epsilon L\qquad\text{or}\qquad|A_i^*+A_j^*|\ge(2-\epsilon)L,
  \end{equation*}
  where $L=p/q$. Now let $S=\{i\in\mathbb{Z}/q\mathbb{Z}:|A_i^*|>\epsilon L\}$. We will call the set $A$ \emph{exceptional} if $|S|\le(\frac{1}{2}-2\delta)\frac{mq}{p}$ and non-exceptional otherwise. Now we need to check that those two classes actually satisfy postulated properties.
  
  The number of exceptional subsets can be estimated as follows: first we choose a prime $q\le\delta^2 p$, then we choose a set $S\subset\mathbb{Z}/q\mathbb{Z}$ of size at most $(\frac{1}{2}-2\delta)\frac{mq}{p}$; there are at most $2^q$ ways of doing that. Having chosen $S$, we specify $A^*$ by choosing $A^*\cap S'$ and $A^*\setminus S'$, where $S'=\bigcup_{i\in  S}I_i$. We take into account that $|A^*\setminus S'|\le \epsilon p$ to get that the number of choices of $A^*$ is bounded by 
  \begin{equation*}
   \sum_{q\le\delta^2 p}\left(2^q\sum_{j\le\epsilon p}\binom{p}{j}\binom{\lfloor (\frac{1}{2}-\delta)(1+\frac{q}{p})m\rfloor}{k-j}\right).
  \end{equation*}
  The bound $|S'|\le(\frac{1}{2}-2\delta)(1+\frac{q}{p})m$ comes from the fact the $S'$ is a union of at most $(\frac{1}{2}-\delta)\frac{mq}{p}$ sets of size at most $\frac{p}{q}+1$. Since $A$ is a dilate of $A^*$, the bound for the number of exceptional subsets is $p$ times the above quantity. Using the estimates from the appendix, we get the claimed bound.
  
  Now let us turn our attention to non-exceptional sets. Suppose that for a set $A$ the regularity lemma gave us a set $S$ of size $|S|>(\frac{1}{2}-2\delta)\frac{mq}{p}$. We would like to show that $S$ is Freiman 2-isomorphic to a set of integers and satisfies the assumption of Theorem \ref{pop}. To prove the former, note that by definition of $\epsilon$-regularity there has to be at least one pair $i,j$ with $|A^*_i+A^*_j|\ge(2-\epsilon)p/q$. The set $A^*_i+A^*_j$ is contained both in $I_{i+j}\cup I_{i+j+1}$ and in $A^*+A^*\subset\{2\lambda, 3\lambda,\ldots,2N\lambda\}$. They are progressions satisfying the assumptions of Lemma \ref{int-prog}; in that case we can argue that the intersection $(I_{i+j}\cup I_{i+j+1})\cap\{2\lambda, 3\lambda,\ldots,2N\lambda\}$ is an interval of length at least $(2-\epsilon)p/q$. Now let $\lambda'$ be an integer less than $\frac{p}{2}$ in absolute value and satisfying the congruence $\lambda\lambda'\equiv 1\pmod{p}$. It is easy to see that 
  \begin{equation*}
   |\lambda'|\le\frac{2N-2}{(2-\epsilon)p/q-1}\le \frac{q}{7}
  \end{equation*}
  since we have a progression of length $(2-\epsilon)p/q$ and common difference $\lambda'$ contained in the interval $[2,N]$. Note also that if $i\in S$, then the interval $I_i$ has nonempty intersection with $\{\lambda, 2\lambda,\ldots,N\lambda\}$. Therefore $\lambda'I_i$ intersects $\{1,2,\ldots,N\}$, but $\lambda'I_i$ is itself contained in an interval of length $|\lambda'|p/q\le\frac{p}{7}$. This interval has to be contained in $[-\frac{p}{7},N+\frac{p}{7}]\subset[-\frac{p}{7},\frac{2p}{7}]$ by the intersection property. Therefore $|\lambda'|i\in[-\frac{q}{7},\frac{2q}{7}]$ as an element of $\mathbb{Z}/q\mathbb{Z}$ and so all of $S$ has to be contained in an arithmetic progression of length less than $\frac{q}{2}$, as required.
  
  We also need to get the bound for $\sum_{y\in\mathbb{Z}/q\mathbb{Z}}\min(1_S*1_S(y),t)$ to be in position to use Theorem \ref{pop}. To do that, set $t=\lfloor 2^{-3}\delta^2 q\rfloor$ and let $T$ be the set of all $y\in\mathbb{Z}/q\mathbb{Z}$ for which there exist $i,j\in S$ with $|A_i^*+A_j^*|\ge(2-\epsilon)p/q$. The sumset $A_i^*+A_j^*$ is contained in $I_{i+j}+I_{i+j+1}$, which allows us to write
  \begin{multline*}
   \frac{(2-\epsilon)p|T|}{q}\le\sum_{y\in T} |(A+A)\cap (I_y\cup I_{y+1})|=\\
   =\sum_{y\in T+\{0,1\}}|(A+A)\cap I_y|+\sum_{y\in T\cap(T+1)}|(A+A)\cap I_y|\le\\
   \le|A+A|+\sum_{y\in T\cap(T+1)}|I_y|\le m+|T\cap(T+1)|\left(\frac{p}{q}+1\right).
  \end{multline*}
  Multiplying the above by $\frac{q}{p}$ we get $(2-\epsilon)|T|\le |T\cap(T+1)|(1+\frac{q}{p})+\frac{mq}{p}$, which leads to
  \begin{multline*}
   |T+\{0,1\}|=2|T|-|T\cap(T+1)|\le\frac{q(m+|T\cap(T+1)|)}{p}+\epsilon|T|\le\\
   \le\frac{q}{p}(m+q+\epsilon p)\le\frac{q(m+\delta^2 p)}{p}\le\frac{(1+\delta)mq}{p}.
  \end{multline*}
  This will be more useful later, but at the moment the most important thing for us is that $|T|\le\frac{(1+\delta)mq}{p}$. On the other hand, every $y\in(\mathbb{Z}/q\mathbb{Z})\setminus T$ corresponds to $1_S*1_S(y)$ ``bad'' pairs $(i,j)$ , for which $\min(|A_i^*|,|A_j^*|)>\epsilon p/q$, yet $|A_i^*+A_j^*|<(2-\epsilon)p/q$. The number of those does not exceed $\epsilon q^2$ and therefore we have
  \begin{multline*}
   \sum_{y\in\mathbb{Z}/q\mathbb{Z}}\min(1_S*1_S(y),t)=\sum_{y\in T}\min(1_S*1_S(y),t)+\sum_{y\not\in T}\min(1_S*1_S(y),t)\le\\
   \le t|T|+\epsilon q^2\le\frac{mqt}{p}\left(1+\delta+\frac{\epsilon pq}{mt}\right).
  \end{multline*}
  Combining the inequalities $t\ge2^{-4}\delta^2q$, $m\ge 2^{-3}\delta p$ and $\epsilon=2^{-7}\delta^4$, we can bound the above by $(1+2\delta)\frac{mqt}{p}$.  Recalling that $|S|\ge(\frac{1}{2}-2\delta)\frac{mq}{p}$, we see that we can use Theorem \ref{pop} if $\delta$ is sufficiently small. Indeed, we can rewrite the bounds as 
  \begin{equation*}
   \sum_{y\in\mathbb{Z}/q\mathbb{Z}}\min(1_S*1_S(y),t)\le\left(2+\frac{12\delta}{1-4\delta}\right)t|S|.
  \end{equation*}
  Since $\frac{1}{1-4\delta}\le\frac{25}{24}$, we can argue that all of $S$, perhaps except $\frac{5t}{2}$ elements, is contained in a progression $Q\subset\mathbb{Z}/q\mathbb{Z}$ of length at most $|S|(1+25\delta)+6t$. Now we use the assumption on $|T+\{0,1\}|$ to say something about the common difference of $Q$. By Pollard's Theorem applied to the set $S\cap Q$ we know that
  \begin{equation*}
   \sum_{y\in Q+Q}\min(1_S*1_S(y),t)\ge t(2|S\cap Q|-t)\ge t(2|S|-6t).
  \end{equation*}
  Subtracting off those elements of $Q+Q$ that are not in $T$, we get
  \begin{equation*}
   \sum_{y\in (Q+Q)\cap T}\min(1_S*1_S(y),t)\ge t(2|S|-6t)-\epsilon q^2.
  \end{equation*}
  This means that we have $|(Q+Q)\cap T|\ge 2|S|-6t-\frac{\epsilon q^2}{t}$, or in other words $|(Q+Q)\setminus T|\le 2(25\delta |S|+6t)+6t+\frac{\epsilon q^2}{t}=50\delta|S|+18t+\frac{\epsilon q^2}{t}$. By the estimate we already know $\sum_{y\in\mathbb{Z}/q\mathbb{Z}}\min(1_S*1_S(y),t)\le(1+\delta)\frac{mqt}{p}$ and Pollard's Theorem we see that $2|S|-t\le(1+\delta)\frac{mq}{p}\le\frac{60 mq}{50p}$, so $2|S|\le\frac{61 mq}{50p}.$ Also, $18t\le 18\frac{\delta mq}{p}$ and $\frac{\epsilon q^2}{t}\le\frac{\delta mq}{p}$ , which altogether gives $|(Q+Q)\setminus T|\le\frac{100\delta mq}{p}$. Now we will examine how the set $Q+Q$ behaves under addition of $\{0,1\}$. The elements of $Q+Q$ inside $T$ will form a subset of $T+\{0,1\}$ of size at most $\frac{(1+\delta)mq}{p}$. The part outside $T$ will get at most doubled, so it will have size at most $\frac{200\delta mq}{p}$. Therefore $|Q+Q+\{0,1\}|\le\frac{(1+201\delta)mq}{p}$. But $Q+Q$ is a progression of length at least $2|S\cap Q|-1\ge 2|S|-5t-1\ge\frac{mq}{p}
(1-4\delta-6\delta)=\frac{(1-10\delta)mq}{p}$ and therefore adding $\{0,1\}$ to it produces at most $\frac{211\delta mq}{p}$ new elements. By dilating, we can assume that $Q$ is an interval and the set we are adding is $\{0,d\}$ for some $d$. But then we see that $d\le\frac{211\delta mq}{p}$. Now consider the dilation by $d$ inside $\mathbb{Z}/p\mathbb{Z}$; then the intervals $I_y$ for $y\in Q$ become progressions of common difference $d$, and $Q$ corresponds to an ``interval'' of such progressions. Their union is almost an interval itself --- the only problem being near endpoints, where the progressions do not necessarily start at the points we like. We can compensate this by adding at most $\frac{p}{q}$ points for each ``residue class $\pmod d$'' to get a genuine interval. The quotation marks mean that we are working $\pmod p$, so technically we cannot consider residue classes modulo other numbers, but we have proved that we use only half of the space, so we can pretend we work in the integers. In the  end 
we get an interval $P$ of length at most $(|Q|+2d)(\frac{p}{q}+1)$ containing almost all of $A$; estimating that gives 
  \begin{equation*}
   (|Q|+2d)(\frac{p}{q}+1)\le\frac{mq}{p}(\frac{1+201\delta}{2}+422\delta)(1+\delta^2)\frac{p}{q}\le\frac{(1+1200\delta)m}{2}.
  \end{equation*}
  Now note that the elements of $A+A$ outside $P+P$ are at worst the ones in $I_y$ with $y\not\in Q+Q$. Each $y\in T$ corresponds to at least $(1-\epsilon)\frac{p}{q}-1$ elements of $A+A$, so since $|(Q+Q)\cap T|\ge 2|S|-6t-\frac{\epsilon q^2}{t}\ge \frac{mq}{p}(1-4\delta-6\delta-\delta)=\frac{1-11\delta}{mq}{p}$, we see that the intersection $(A+A)\cap (P+P)$ has at least $\frac{1-11\delta}{mq}{p}((1-\epsilon)\frac{p}{q}-1)\ge(1-12\delta)m$ elements. That means that $|(A+A)\setminus (P+P)|\le12\delta m\le 24\delta N$, as claimed.
 \end{proof}
 
 \section{Proof of Theorem \ref{prob}}
 
 In this section we use the prevous results to get Theorem \ref{prob}; first let us note that it can be easily reduced to the following statement.
 
 \begin{prp}\label{recurrence}
  There exist absolute constants $C_0$, $\epsilon_0$, $k_0$ such that the following is true. Suppose that $A\subset\mathbb{N}$ is a set chosen randomly by picking each element of $\mathbb{N}$ independently with probability $\frac{1}{2}$. Then for every $k>k_0$ we have the inequality
  \begin{equation*}
   \mathbb{P}(|\mathbb{N}\setminus (A+A)|\ge k\text{ and }1\in A)\le C_0(2+\epsilon_0)^{-k/2}.
  \end{equation*}
 \end{prp}

 Comparing this to what we are trying to prove in the end, this statement says that for a set $A$ satisfying $|\mathbb{N}\setminus (A+A)|\ge k$ it is exponentially (in $k$) unlikely to contain $1$.
 
 \begin{proof}[Proof of Theorem \ref{prob} assuming Proposition \ref{recurrence}]
  Let $A\subset\mathbb{N}$ be a random subset. Note that the conditional distribution of $A$ on the event $1\not\in A$ is exactly the same as the initial distribution of $A+1$. This, and the fact that $\mathbb{P}(1\in A)=\frac{1}{2}$, allows us to write (for each $k\ge 2$):
  \begin{multline*}
   p_k-p_{k-2}=2^{k/2}\cdot\mathbb{P}(|\mathbb{N}\setminus (A+A)|\ge k)-2^{(k-2)/2}\cdot\mathbb{P}(|\mathbb{N}\setminus (A+A)|\ge k-2)\\
   =2^{k/2}(\mathbb{P}(|\mathbb{N}\setminus (A+A)|\ge k)-\tfrac{1}{2}\mathbb{P}(|\mathbb{N}\setminus ((A+1)+(A+1))|\ge k)=\\
   =2^{k/2}(\mathbb{P}(|\mathbb{N}\setminus (A+A)|\ge k)-\mathbb{P}(|\mathbb{N}\setminus (A+A)|\ge k\text{ and }1\not\in A))=\\
   =2^{k/2}\cdot\mathbb{P}(|\mathbb{N}\setminus (A+A)|\ge k\text{ and }1\in A).
  \end{multline*}
  The above quantity is obviously nonnegative, which makes both sequences $\{p_{2k}\}$ and $\{p_{2k+1}\}$ increasing. Proposition \ref{recurrence} allows us to say they are bounded. Indeed, if $k>k_0$, then
  \begin{multline*}
   p_k-p_{k-2}=2^{k/2}\cdot\mathbb{P}(|\mathbb{N}\setminus (A+A)|\ge k\text{ and }1\in A)\le\\
   \le 2^{k/2}\cdot C_0(2+\epsilon_0)^{-k/2}=C_0\left(\frac{2}{2+\epsilon_0}\right)^{k/2}.
  \end{multline*}
  Let $\lambda=\frac{2}{2+\epsilon_0}<1$. Summing the above inequalities we see that for any $k>k_0$ we get
  \begin{equation*}
   p_k-p_{k_0}\le \sum_{s=k_0/2}^\infty C_0\lambda^s=\frac{C_0\lambda^{k_0/2}}{1-\lambda}<\infty.
  \end{equation*}
  Actually the above is true only for $k$ of the same parity as $k_0$; for the remaining values we simply replace all instances of $k_0$ with $k_0+1$.
 \end{proof}

 For the rest of this section we will focus on proving Proposition \ref{recurrence}.
 
 \begin{proof}[Proof of Proposition \ref{recurrence}]
  Let $\delta$ be a small quantity and let $k>N_0(\delta)$ from the statement of the Proposition \ref{counting}. Consider the set $X=A\cap[10k]$. First, following Green and Morris, let us estimate the probability that $A$ misses one of the elements greater than $10k$. For each such element $m$ we have at least $\lfloor\frac{m}{2}\rfloor$ pairs of (not necessarily distinct) natural numbers $u,v$ with $u+v=m$; the probability that $u,v\in A$ is at least $\frac{1}{4}$. Therefore the probability that $m\not\in A+A$ is bounded by $(\frac{3}{4})^{(m-1)/2}$ and the total contribution for numbers greater than $10k$ is at most
  \begin{equation*}
   \sum_{m=10k+1}^\infty\left(\frac{3}{4}\right)^{(m-1)/2}<2^{-k}.
  \end{equation*}
  We consider this a quantity less than $C_0(2+\epsilon_0)^{-k/2}$; our goal is to divide the set of admissible events into classes with probability of each being bounded by this expression.
  
  From this point on we can assume that $A$ contains all the numbers greater than $10k$ and consequently $\mathbb{N}\setminus (A+A)=[10k]\setminus(X+X)$. Let us estimate the probability that $|X|\le10\delta k$; since $X$ is uniformly distributed among all subsets of $[10k]$, we can estimate it by $2^{10k(H(\delta)-1)}$. If $\delta$ is small enough, this implies the claimed bound.
  
  Now assume that $|X|>10\delta k$ and estimate the probability that $X$ is exceptional (according to the statement of Proposition \ref{counting}) and obeys the inequality $|[10k]\setminus(X+X)|\ge k$. This is bounded by
  \begin{equation*}
   2^{-10k}\sum_{k'>10\delta k}\sum_{m=2k'-1}^{19k}2^{\frac{m}{2}H(\frac{2k'-1}{m})-10\delta k}.
  \end{equation*}
  The estimate $m\le 19k$ comes from the fact that $X+X$ misses at least $k$ points from $[10k]\subset [20k]$. Bounding each term crudely, i.e. using $H(x)\le 1$, $m\le 19k$, we get the bound of $200k^2 2^{-k/2-10\delta k}$, again as good as we need.
  
  \begin{equation*}
   \sum_{k'=0}^{m/2}2^{10\delta k}\binom{m/2}{k'}.
  \end{equation*}
  If we restrict the range of $m\le (19-30\delta)k$, the above expression divided by $2^{10k}$ is still bounded as we need. Therefore we can assume that in fact $|X+X|\ge (19-30\delta)k$.

  Our bound gives us that the corresponding progression $P$ has size at least $(1-10\delta)\frac{m}{2}\ge(\frac{19}{2}-10^3\delta)k$. Therefore it has common difference $1$. Suppose now that $X+X$ contains at least $2\cdot10^4\delta k$ elements less than $k/2$. We know that $(X+X)\setminus (P+P)\le 12\delta m\le 240\delta k$, which means that the least element of $P$ is at most $(1-19000\delta) k/2$. On the other hand, since $|P|\le\frac{1-1200\delta}{2}m\le (\frac{19}{2}-12000\delta)k,$ so there are at least $7000\delta k$ elements of $(10k,20k]$ not in $P+P$, which belong to $A+A$ anyway.
  But out of them only $240\delta k$ can belong to $X+X$, so that leaves over $6000\delta k$ elements of $(A+A)\setminus (X+X)$ in $[20k]$. Adding to that $m\ge (19-30\delta)k$ elements of $X+X$ we see that $A+A$ in fact misses less than $k$ elements, so this case cannot hold.
  
  Suppose now that $A$ (equivalently: $X$) contains less than $2\cdot10^4\delta k$ elements less than $k/2$. So far we have not use the fact that $1\in A$. We are going to do this now. Let $B=X\cap(0,\frac{k+1}{2}]$, $C=X\cap(\frac{k+1}{2},k]$ and $D=X\cap(k,+\infty)$. Clearly $A+A\supset (1+C)\cup(C+C)\cup(D+D)$ and those subsets are disjoint. So certainly $|\mathbb{N}\setminus(A+A)|\ge k$ if 
  \begin{equation*}
   |\mathbb{N}\setminus(D+D-2k)|\ge k-(2k-|C|-|C+C|)=|C|+|C+C|-k.
  \end{equation*}
 By Green-Morris estimate \cite[Theorem 1.3]{green-15} we can assume that the probability of the latter is bounded by $C_\epsilon(2-\epsilon)^{-(|C|+|C+C|-k)}$ for any $\epsilon>0$ and a suitable constant $C_\epsilon>0$. Note that the probability cannot exceed $1$, so the trivial bound is actually better if $|C|+|C+C|<k$. Putting everything together, the total probability that $|\mathbb{N}\setminus(A+A)|\ge k$ is bounded by
  \begin{equation*}
   2^{-k/2}\sum_B 2^{-k/2}\sum_C C_\epsilon\cdot\min (1,(2-\epsilon)^{-(|C|+|C+C|-k)/2}).
  \end{equation*}
  Now by Green-Morris bound on the number of sets with small sumset \cite[Proposition 3.1]{green-15} we can divide the inner sum into classes depending on the size of $|C|$ and $|C+C|$. This way we get a bound of
  \begin{equation*}
   C_\epsilon\cdot2^{-k}\sum_B \sum_{l,m\le k} 2^{\delta k}\binom{m/2}{l}\cdot\min (1,(2-\epsilon)^{-(l+m-k)/2}).
  \end{equation*}
  Now we have the upper bound on the size of $B$, which turns the outer sum into additional coefficient $2^{kH(4\cdot 10^4\delta)/2}$. The only way the above expression could fail to be bounded as we need is if we could find among the expressions $\binom{m/2}{l}\cdot\min (1,(2-\epsilon)^{-(l+m-k)/2})$ one that is greater than $(2-\eta)^{k/2}$ for some small value of $\eta$. This firstly requires $m$ to be close to $k$ as the binomial coefficient $\binom{m/2}{l}\le2^{m/2}$ has to be at least $(2-\eta)^{k/2}$. Also, we need to have $(2-\epsilon)^{-(l+m-k)/2}\ge(1-\frac{\eta}{2})^{k/2}$, which requires $l+m$ to be not much bigger than $k$. Those two in turn imply that $l$ is small compared to $k$, in which case $\binom{m/2}{k}$ does not exceed $(2-\eta)^{k/2}$. This indicates that our initial assumption was false and our bound in fact is always satisfied.
 \end{proof}
 
 \appendix
 
 \section{Estimates of binomial coefficients}
 
 In the appendix we are going to prove some useful inequalities concerning binomial coefficients. Before we do that, let us define the binary entropy function $H:[0,1]\to\mathbb{R}$ as $H(t)=t\log_2\frac{1}{t}+(1-t)\log_2\frac{1}{1-t}$. Note that it does not quite make sense if $t=0$ or $t=1$, so we extend $H$ continuously by setting $H(0)=H(1)=0$. One can easily prove that $0\le H(t)\le 1$ for all $t\in [0,1]$ with the only extremal values being $H(0)=H(1)=0$ and $H(\frac{1}{2})=1$. Also, $H$ is continuous, increasing on $[0,\frac{1}{2}]$, decreasing on $[\frac{1}{2},1]$, and obeys the equality $H(t)=H(1-t)$. By calulating the second derivative, one can see that $H$ is concave, which together with previous observations leads to the triangle inequality $H(|x+y|)\le H(|x|)+H(|y|)$ valid for all $x,y$ for which all the expressions make sense.
  
 \begin{lm}\label{est}
  Let $n,k\ge 0$ be integers and let $\delta\in [0,\frac{1}{2}]$ be a real number. Then the following inequalities are true:
  \begin{gather*}
   \frac{2^{nH(\frac{k}{n})}}{n+1}\le \binom{n}{k}\le 2^{nH(\frac{k}{n})},\\
   \sum_{j=0}^{\lfloor\delta{n}\rfloor}\binom{n}{j}\le 2^{nH(\delta)}.
  \end{gather*}
 \end{lm}
 
 \begin{proof}
  Let us begin with the last inequality. Take a set of size $n$ and choose a subset of it at random by picking each element independently with probability $\delta$. Then for any $j\le\delta n$ the probability that a given $j$-element subset is chosen is $\delta^j(1-\delta)^{n-j}$. Since $\delta\le 1-\delta$, we can bound it from below by
  \begin{equation*}
   \delta^j(1-\delta)^{n-j}\ge\delta^{\delta n}(1-\delta)^{(1-\delta)n}=2^{-nH(\delta)}.
  \end{equation*}
  But the total probability cannot exceed $1$, so the number of all those sets, equal to $\sum_{j=0}^{\lfloor\delta{n}\rfloor}\binom{n}{j}$, has to be at most $2^{nH(\delta)}$.
  
  Turning to the first two inequalities, assume without loss of generality that $k\le\frac{n}{2}$; we can do that as we can always replace $k$ with $n-k$, given that $\binom{n}{k}=\binom{n}{n-k}$ and $H(\frac{k}{n})=H(1-\frac{k}{n})=H(\frac{n-k}{n})$.
  In this case let $\delta=\frac{k}{n}\le\frac{1}{2}$ and consider again the same random experiment. Writing the total probability as a sum of probabilities of choosing a particular set, we get
  \begin{equation*}
   \sum_{j=0}^n\binom{n}{j}\delta^j(1-\delta)^{n-j}=1.
  \end{equation*}
  Note that the expression $\binom{n}{k}2^{-nH(\frac{k}{n})}$ is one of the summands in the above sum, and by comparing two consecutive ones we can check that is actually the largest out of $n+1$ summands. Therefore $\frac{1}{n+1}\le\binom{n}{k}2^{-nH(\frac{k}{n})}\le 1$, as needed.
 \end{proof}

\end{document}